\documentclass[11pt,leqno]{article}
\usepackage{amsmath,amssymb,amsthm,nicefrac,mathrsfs}
\usepackage[colorlinks]{hyperref}
\usepackage{graphicx}
\usepackage[caption=false]{subfig}
\usepackage{tabu}
\usepackage{bbm}
\usepackage{hyperref}
\usepackage{indentfirst}
\usepackage{mathrsfs}
\usepackage{physics}
\usepackage{calc}
\usepackage{dutchcal}
\usepackage{mathalpha}
\usepackage{mathtools}
\usepackage{lipsum}
\usepackage{verbatim}
\usepackage[a4paper, margin=4cm, footskip=1cm]{geometry}
\usepackage{caption}
\usepackage[T1]{fontenc}
\usepackage{booktabs}
\usepackage{siunitx}
\usepackage{scalerel,stackengine}

\stackMath
\newcommand\reallywidehat[1]{%
	\savestack{\tmpbox}{\stretchto{%
			\scaleto{%
				\scalerel*[\widthof{\ensuremath{#1}}]{\kern-.6pt\bigwedge\kern-.6pt}%
				{\rule[-\textheight/2]{1ex}{\textheight}}
			}{\textheight}%
		}{0.5ex}}%
	\stackon[1pt]{#1}{\tmpbox}%
}

\providecommand{\keywords}[1]
{
	\small	
	\textbf{\textit{Keywords---}} #1
}

\newcommand{\Addresses}{{
		\bigskip
		\footnotesize

		\noindent 	NGARTELBAYE GUERNGAR \\
		\textsc{Department of Mathematics, University of North Alabama,
			Florence, AL 35832}\\
		\textit{E-mail address}: \texttt{nguerngar@una.edu}\\
		\textit{URL}: \texttt{\url{	http://buildingthepride.com/faculty/nguerngar}}
		
		\medskip

		\noindent ERKAN NANE \\
		\textsc{Department of Mathematics and Statistics, Auburn University,
			Auburn, AL 36849}\\
		\textit{E-mail address}: \texttt{ezn0001@auburn.edu}\\
		\textit{URL}: \texttt{\url{http://www.auburn.edu/~ezn0001}}
		
		\medskip

}}

\newtheorem{theorem}{Theorem}[section]

\newtheorem{proposition}[theorem]{Proposition}

\newtheorem{assumption}[theorem]{Assumption}
\newtheorem{remark}[theorem]{Remark}

\numberwithin{equation}{section}

{%

	\begin{enumerate}}%
	{\end{enumerate}
}

\def\be{\begin{equation}}
	\def\ee{\end{equation}}
\def\ba{\begin{aligned}}
	\def\ea{\end{aligned}}
\def\bes{\begin{equation*}}
	\def\ees{\end{equation*}}

\def\bc{\begin{cases}}
	\def\ec{\end{cases}}
\numberwithin{equation}{section}
\everymath{\displaystyle}

\author{Ngartelbaye Guerngar\\
	University of North Alabama\\
	\and Erkan Nane\\ Auburn University\\ 
}
\title{{ Space-time fractional SPDEs with locally Lipschitz coefficients: well-posedness}	
	\date{}
}

\begin{document}
	\maketitle
	
	
	

	
	\begin{abstract}
		\noindent In this article, we study the space-time SPDE 
		$$ \partial_t^\beta u=-(-\Delta)^{\alpha/2} u+I_t^{1-\beta}[b(u)+\sigma(u)\dot{W}],$$
		where $u=u(t,x)$ is defined for $(t,x)\in\mathbb{R}_+\times \mathbb{R},$ $\beta\in(0,1), \alpha\in(0,2)$ and $\dot{W}$ denotes a space-time white noise. 
		It has long been conjectured that this equation has a unique solution with finite moments under the minimal assumptions of locally Lipschitz coefficients $b$ and $\sigma$ with linear growth. 
		We prove that this SPDE is well-posed  under the assumptions that the initial condition $u_0$ is bounded and measurable, and the functions $b$ and $\sigma$ are locally Lipschitz and have at-most linear growth and some conditions on the Lipschitz constants on the truncated versions of $b$ and $\sigma$.  Our results generalize the work of Foondun et al.(2025) to a space-time fractional setting.
		
	\end{abstract}
	
	
	\newpage
	\section{Introduction}
	\indent 
	
	In this paper, we study the well-posedness of the following space-time stochastic fractional diffusion 
	
	\begin{equation}\label{MainEq}
		\begin{split}
			\begin{cases}
				\partial_t^\beta u(t,x)=&-(-\Delta)^{\alpha/2}u(t,x)+I_t^{1-\beta}\big[b(u(t, x))+\sigma(u(t,x))\dot{W}(t,x)\big],  \ t>0,\  x\in \mathbb{R}^d, \\
				\ \  u(0,\cdot)=&u_0(\cdot),
			\end{cases} 
		\end{split}
	\end{equation}
	with $\alpha\in(0,2]$, $b$ and $\sigma$ is a locally Lipschitz function with at-most linear growth.  $\dot{W}$ represents the space-time white noise.
	$-(-\Delta)^{\alpha/2}$ is the fractional Laplacian. The initial condition $u_0$ is a function satisfying some conditions (to be specified later). $\partial_t^\beta$ is the Caputo fractional differential operator for $\beta\in(0,1)$, defined by:
	\begin{equation*}
		\partial_t^\beta f(t)=\frac{1}{\Gamma(1-\beta)}\int_0^t \frac{f'(s)}{(t-s)^\beta} ds,
	\end{equation*}
	and $I_t^\gamma$ is the Riemann-Liouville fractional integral of order $\gamma>0$, defined by
	$$I_t^\gamma f(t)=\frac{1}{\Gamma(\gamma)}\int_0^t(t-s)^{\gamma-1}f(s)ds, \ \text{for}\ t>0,$$  
	with the convention that $I_t^0=\text{Id}$, the identity operator. Here $\Gamma(\cdot)$ is the Euler gamma function. The Riemann-Liouville fractional integral operator is used in Eq. \eqref{MainEq} to properly handle the "derivative" of the random term, resulting in the space-time white noise, see for example \cite[pp. 3303-3304]{MIJNane-2015} for the details.
	Note that when $\beta=1$ and $\alpha=2,$ the fractional derivatives $\partial_t^\beta$ and $-(-\Delta)^{\alpha/2}$ become the first order derivative $\frac{\partial}{\partial t}$ and the Laplacian, $\Delta=\frac{\partial^2}{\partial x_1^2}+\cdots+\frac{\partial^2}{\partial x_d^2} $, respectively. In this case, problem \eqref{MainEq} becomes the classical diffusion problem studied in \cite{FooSanLart} for $d=1$.

	The space-time fractional stochastic partial differential equation \eqref{MainEq} (with some variations) has been studied a lot recently, see for example \cite{GueNan, ChenHuNua,MIJNane-2015, MijNane} for the derivation of this equation and motivation to study it.
	When $\sigma$ is Lipschitz continuous in the spatial variable and uniformly in the time variable, 
	it was shown in \cite{MIJNane-2015} that \eqref{MainEq} is well posed (in the case $b\equiv0)$ provided that $d<\min\big(2,\beta^{-1}\big)\alpha$. The same argument can also be used to show that this remains true if $b$ is also Lipschitz continuous. In this work, we aim to extend the well-posedness result to the case when the Lipschitz condition is weakened to locally Lipschitz with at-most linear growth; more precisely, we show that when $b$ and $\sigma$ are locally Lipschitz with at-most linear growth in their spatial variable, Eq. \eqref{MainEq} has a unique solution with finite moments provided that $d<\min\big(2,\beta^{-1}\big)\alpha$. We use the roadmap of Foondun et al \cite{FooSanLart} to prove our results  with crucial changes mainly  in the use of the bounds of the heat kernel estimates. Our results extend the work of \cite{FooSanLart} to time and space fractional setting. The main idea here is the use of truncation, mainly  point-wise tail probability estimates for the truncated solution.

	\section{Preliminary results}
	
	In this section, we describe the truncation argument mentioned in the Introduction and provide some intermediate results needed for the proof of our main theorem which we provide in the next section.

	We define a solution of \eqref{MainEq} as a predictable random field $u=\{u(t,x)\}_{t>0,x\in\mathbb{R}}$ satisfying the following integral equation
	
	\begin{equation}\label{MildSol}
		u(t,x)=(G_t\star u_0)(x)+\mathcal{I}_b(t,x)+\mathcal{I}_\sigma(t,x),
	\end{equation}
	where
	\begin{equation*}
		(G_t\star u_0)(x)= \int\limits_{\mathbb{R}^d} G_t(y-x)u_0(y)dy,
	\end{equation*}
	\begin{equation}\label{Ib}
		\mathcal{I}_b(t,x)=\int\limits_0^t \int\limits_{\mathbb{R}^d} \ G_{t-s}(y-x)b\big(s, u(s,y)\big)dyds,
	\end{equation}
	\begin{equation}\label{Is}
		\mathcal{I}_\sigma(t,x)=\int\limits_0^t \int\limits_{\mathbb{R}^d} \ G_{t-s}(y-x)\sigma\big(u(s,y)\big)W(dy,ds), 
	\end{equation}
	and $G_{\cdot}(\cdot)$ is the "heat kernel" satisfying
	\begin{equation}\label{G:Bounds-}
		c_1\Bigg(t^{-\beta d/\alpha}\wedge \frac{t^\beta }{|x|^{d+\alpha}}\Bigg)\leq G_t(x)\leq c_2\Bigg(t^{-\beta d /\alpha}\wedge \frac{t^\beta }{|x|^{d+\alpha}}\Bigg),
	\end{equation}
	for some positive constants $c_1$ and $c_2$ and the integral defined \eqref{Is} is understood in the sense of Walsh \cite{Walsh}.
		Using basic properties of the heat kernel $G$, see \cite[Lemma 1]{MIJNane-2015}, we can show that
	\begin{equation}\label{HeatKernel}
		{\|G_r\|}^2_{L^2(\mathbb{R}^d)}=\int\limits_{\mathbb{R}^d}G_r^2( y)dy= C^{\star}r^{-\beta/\alpha} \qquad \text{for every} \ r>0, 
	\end{equation}
	\begin{equation}\label{Const-G}
		C^{\star}= \frac{2\pi^{d/2}}{\alpha\Gamma(d/2)}\frac{1}{(2\pi)^{d}}\int\limits_0^\infty z^{1/\alpha-1}E^2_\beta(-z)dz;
	\end{equation}
	and $E_\beta(z):=\sum\limits_{\ell=1}^\infty\frac{z^\ell}{\Gamma(1+\beta \ell)}$ is the one-parameter Mittag-Leffler function.
	
	\begin{remark}
		For simplicity, we will assume $d=1$ in our calculations for the remainder of this paper, but we would like to point out that our arguments remain valid in higher spatial dimensions $d<\min\big(2,\beta^{-1}\big)\alpha$.
	\end{remark}
	
	For $X\in L^p(\Omega)$, we define the $L^p(\Omega)-$norm as ${\|X\|}_p=\big[\mathbb{E}|X|^p\big]^{1/p}$ for all $p\geq 1$. In addition, for every space-time function $f:\mathbb{R}_+\times\mathbb{R}\rightarrow\mathbb{R}$, we will use the notation $\text{Lip}(f)$ for the optimal Lipschitz constant of $f$, i.e,
	\begin{equation}\label{OptLips}
		\text{Lip}(f)= \sup\limits_{t>0}\sup_{\substack{a,b\in\mathbb{R} \\ a\neq b}} \frac{|f(t,b)-f(t,a)|}{|b-a|}.
	\end{equation}
	It is not hard to see that $f$ is globally Lipschitz continuous in $x$, uniformly in $t$, if and only if $\text{Lip}(f)<\infty.$ Note that \eqref{OptLips} remains valid even if $f$ is a function of the single spatial variable $x,$ provided that we set $f(t,x)=f(x)$. 
	
	The following assumptions on the in initial condition and coefficients $b$ and $\sigma$ will be used in our arguments.
	\begin{assumption}\label{As1}
		$u_0:\mathbb{R}\rightarrow\mathbb{R}$ is non-random, bounded and measurable.   
	\end{assumption}
	
	\begin{assumption}\label{As2}
		The functions $b:\mathbb{R}_+\times\mathbb{R}\rightarrow\mathbb{R}$ and $\sigma:\mathbb{R}_+\times\mathbb{R}\rightarrow\mathbb{R}$ are locally Lipschitz continuous in their space variable with at-most linear growth, uniformly in their time variable, i.e, $0<\text{Lip}_n(b), \text{Lip}_n(\sigma)<\infty$ and $0<L_b, L_\sigma<\infty,$ for all real numbers $n>0$ where, for each space-time function $\psi,$
		\begin{equation}\label{Eq-As2}         L_\psi=\sup\limits_{t>0}\sup\limits_{x\in\mathbb{R}}\frac{|\psi(t,x)|}{1+|x|}, \text{and} \ \text{Lip}_n(\psi)=\sup\limits_{t>0}\sup_{\substack{x,y\in[-n,n] \\ x\neq y}} \frac{|\psi(t,x)-\psi(t,y)|}{|x-y|}.
		\end{equation}
	\end{assumption}
	
	Next, for every number $N>0$, we define the truncated coefficients $b_N:\mathbb{R}_+\times\mathbb{R}\rightarrow\mathbb{R}$ and $\sigma_N:\mathbb{R}_+\times\mathbb{R}\rightarrow\mathbb{R}$ via  
	$$
	\psi_N(t,x)=\psi(t,x)\mathbf{1}_{\{|x|<e^N\}}+\psi(t,-e^N)\mathbf{1}_{\{x<-e^N\}}+\psi(t,e^N)\mathbf{1}_{\{x>e^N\}}, \quad \text{for all}\  t>0.
	$$

	We then define the Lipschitz constants of the truncated coefficients as 
	
	\begin{equation}\label{LN}
		L_{N,b}=\text{Lip}_{\exp(N)}(b) \ \text{and} \  L_{N,\sigma}=\text{Lip}_{\exp(N)}(\sigma).
	\end{equation}
	
	It can be easily shown that $b_N$ and $\sigma_N$ are globally Lipschitz for any fixed $N>0$. In fact
	
	\begin{equation}\label{Trc-As2}
		\sup\limits_{t>0}\sup\limits_{x\in\mathbb{R}}\frac{|b_N(t,x)|}{1+|x|}\leq L_b<\infty, \text{and} \ \sup\limits_{t>0}\sup\limits_{x\in\mathbb{R}}\frac{|\sigma_N(t,x)|}{1+|x|}\leq L_\sigma<\infty,
	\end{equation}
	uniformly in $N>0.$ 
	
	We also add an assumption on the Lipschitz constants of the truncated coefficients $b_N$ and $\sigma_N$. This is needed for the calculations of the tail estimates of the truncated solution. 
	\begin{assumption}\label{As3}
		If $L_\sigma>0,$ then we assume that 
		\begin{equation}\label{Ass3-1}
			\text{L}_{N,\sigma}=\mathcal{o}\Big(N^{\frac{(1-\beta/\alpha)(2-\beta/\alpha)}{2}}\Big) \qquad \text{L}_{N,b}/\text{L}_{N,\sigma}^{\frac{2}{1-\beta/\alpha}}=\mathcal{O}(1) \qquad \text{as} \ N\rightarrow\infty.
		\end{equation}
		If $\sigma$ is bounded, then we assume that 
		\begin{equation}\label{Ass3-2}
			\text{L}_{N,\sigma}=\mathcal{o}\Big(e^{N(1-\beta/\alpha)}\Big) \qquad \text{L}_{N,b}/\text{L}_{N,\sigma}^{\frac{2}{1-\beta/\alpha}}=\mathcal{O}(1) \qquad \text{as} \ N\rightarrow\infty.
		\end{equation}
	\end{assumption}
	
	It has long been conjectured that Theorem \ref{MainTh} (in the next paragraph) should hold under the minimal assumptions of locally Lipschitz coefficients $b$ and $\sigma$ with linear growth. We examine this in the next result, which is also our main result in this paper.

	\begin{theorem}\label{MainTh}
		Under Assumptions \ref{As1}, \ref{As2} and \ref{As3}, Eq. \eqref{MainEq} has a unique random field solution satisfying
		\begin{equation*}
			\sup\limits_{t\in(0,T]}\sup\limits_{x\in\mathbb{R}}\mathbb{E}\big(|u(t,x)|^k\big)<\infty \ \text{for all} \ T>0 \ \text{and} \ k\geq 1.
		\end{equation*}
	\end{theorem}
	This Theorem will be proved in Section \ref{Proof-Thm}.
	
	We now consider the truncated version of \eqref{MainEq}:
	\begin{equation}\label{Trc-spde}
		\partial_t^\beta u_N(t,x)=-(-\Delta)^{\alpha/2}u_N(t,x)+I_t^{1-\beta}\big[b_N\big(t, u_N(t,x)\big)+\sigma_N\big(t,u_N(t,x)\big)\dot{W}(t,x)\big],  \ t>0,\  x\in \mathbb{R}
	\end{equation}
	subject to the initial condition   $u_N(0,x)=u_0(x)$. Following standard theory by \cite{Dalang, Walsh, MIJNane-2015}, Eq. \eqref{Trc-spde} has a unique predictable mild solution satisfying 
	\begin{equation}\label{Trc-Cond}
		\sup\limits_{t\in(0,T]}\sup\limits_{x\in\mathbb{R}}\mathbb{E}\big(|u_N(t,x)|^k\big)<\infty \ \text{for all} \ N, T>0 \ \text{and} \ k\geq 1.
	\end{equation}
	Note that, Eq. \eqref{Trc-spde} is a short-hand for the random integral equation
	\begin{equation}\label{Trc-MildSol}
		u_N(t,x)=(G_t\star u_0)(x)+\mathcal{I}_{b_N}^N(t,x)+\mathcal{I}_{\sigma_N}^N(t,x),
	\end{equation}
	where
	
	\begin{equation}\label{IbN}
		\mathcal{I}_{b_N}^N(t,x)=\int\limits_0^t \int\limits_{\mathbb{R}} \ G_{t-s}(y-x)b_N\big(s, u_N(s,y)\big)dyds,
	\end{equation}
	\begin{equation}\label{IsN}
		\mathcal{I}_{\sigma_N}^N(t,x)=\int\limits_0^t \int\limits_{\mathbb{R}} \ G_{t-s}(y-x)\sigma_N\big(s,u_N(s,y)\big)W(dy,ds).
	\end{equation}
	
	In the next two Propositions, we use the methodology of \cite{FooSanLart} and adapt it to our case with crucial changes in the estimates of the heat kernel to find moment and tail estimates for the truncated solution. We describe the details in order to provide the explicit constants and parameter dependencies of the bounds.
	\begin{proposition}\label{Prob4}
		\begin{enumerate}
			\item In one hand, if $L_\sigma>0,$ then $$ \sup\limits_{N>0}\sup\limits_{x\in\mathbb{R}}\mathbb{E}\big(|u_N(t,x)|^k\big)\leq C_0^k\exp{4(C_{\#}L_\sigma)^{\frac{2}{1-\beta/\alpha}}k^{1+\frac{1}{1-\beta/\alpha}} t}
			$$
			uniformly for all $t>0$ and $k\geq \max\Big(2, L_b^{1-\beta/\alpha}L_\sigma^{-2}\Big)$. Here, $C_0:=4({\|u_0\|}_{L^\infty(\mathbb{R})}+1) $ and $C_{\#}:=4\sqrt{C^{\star}\Gamma(1-\beta/\alpha)}$, where $C^{\star}$ is defined in \eqref{Const-G}.
			\item On the other hand, if $\sigma\in L^\infty (\mathbb{R}_+\times\mathbb{R})$, then
			$$
			\sup\limits_{N>0}\sup\limits_{x\in\mathbb{R}}\mathbb{E}\big(|u_N(t,x)|^k\big)\leq C_{\star}^k e^{2kL_b t}\Big({\|u_0\|}_{L^\infty(\mathbb{R})}+{\|\sigma\|}_{L^\infty(\mathbb{R_{+}}\times\mathbb{R})}+1\Big)^kk^{k/2},
			$$
			
			uniformly for all $t>0$ and $k\geq 2$. Here, $C_{\star}:=\max\Bigg(2,\frac{4C_{\alpha, \beta, \gamma}\sqrt{C^{\star}}}{\sqrt{1-\beta/\alpha}}\Bigg)$ and $C^{\star}$ is again defined in \eqref{Const-G} and $C_{\alpha, \beta, \gamma}$ is a positive constant depending on $\alpha, \beta$ and $\gamma$.
		\end{enumerate} 
	\end{proposition}
	\begin{proof}
		\begin{enumerate}
			\item \underline{$L_\sigma>0$}:
		\end{enumerate}
		Choose and fix, $N, t>0$ and $x\in\mathbb{R.}$ Using \eqref{Trc-MildSol}, it follows that
		\begin{equation}\label{Trc-norm}
			{\|u_N(t,x)\|}_k\leq {\|u_0\|}_{L^{\infty}(\mathbb{R})}+I_1+I_2, 
		\end{equation}
		where $I_1={\|\mathcal{I}_b^N(t,x)\|}_k$ and $I_2={\|\mathcal{I}_\sigma^N(t,x)\|}_k$. We then proceed to estimate $I_1$ first. To this aim, we apply \eqref{Eq-As2} and Minkowski's inequality for integrals together with the fact that 
		
		$${\|b_N\big(s, u_N(s,y)\big)\|}_k\leq L_b\big(1+{\|u_N\big(s,y)\big\|}_k\big) \qquad (\text{which follows from \eqref{Trc-As2}})$$ to get
		\begin{align*}
			I_1\leq &\int\limits_0^tds \int\limits_{\mathbb{R}} dy \ G_{t-s}(y-x){\|b_N\big(s,u_N(s,y)\big)\|}_k\\
			\leq & L_b\Big[t+\int\limits_0^t \sup\limits_{y\in\mathbb{R}}{\|u_N(s,y)\|}_kds\Big].
		\end{align*}
		Define
		\begin{equation}\label{N-norm}
			\mathcal{N}_{k, \gamma}(Z)=\sup\limits_{t> 0}\sup\limits_{x\in\mathbb{R}}e^{-\gamma t} {\|Z(t,x)\|}_k,
		\end{equation}
		for all space-time random field $Z={\{Z(t,x)\}}_{\substack{t>0 \\ x\in\mathbb{R}}}$ and  real numbers $k\geq 1$ and $\gamma>0.$
		
		Using this notation, we can bound $I_1$ as follows:
		$$I_1\leq L_b\Big[t+\mathcal{N}_{k, \gamma}(u_N)\int\limits_0^t e^{\gamma s}ds\Big]\leq  L_b\Big[t+\frac{e^{\gamma t}}{\gamma}\mathcal{N}_{k, \gamma}(u_N)\Big].$$
		Since $te^{-\gamma t}\leq (e\gamma)^{-1}\leq \gamma^{-1},$ it follows that
		\begin{equation}\label{I1-bound}
			I_1\leq L_b\frac{e^{\gamma t}}{\gamma}\Big[1+\mathcal{N}_{k, \gamma}(u_N)\Big].
		\end{equation}
		Next, using the asymptotically optimal form of the Burkh\"older-Davis-Gundy inequality (see for example \cite{DaViMuLarXiao}), we find an upper bound for $I_2$:
		\begin{equation}\label{I2-sq}
			I_2^2\leq 4k \int\limits_0^t ds\int\limits_{\mathbb{R}} dy \  \big[G_{t-s}(y-x)\big]^2{\|\sigma_N\big(s,u_N(s,y)\big)\|}_k^2.
		\end{equation}
		Now, using the fact that ${\big\|\sigma_N\big(s, u_N(s,y)\big)\big\|}_k^2\leq 2L_\sigma^2\Big(1+{\|u_N\big(s,y)\big\|}_k^2\Big)$ (which also follows from \eqref{Trc-As2}), we get
		$$
		I_2^2\leq 8k L_\sigma^2\int\limits_0^t ds\int\limits_{\mathbb{R}} dy \  \big[G_{t-s}(y-x)\big]^2\big(1+{\|u_N(s,y)\|}_k^2\big).
		$$

		Thus, using \eqref{HeatKernel}, we have
		\begin{align*}
			I_2^2\leq& 8k L_\sigma^2C^{\star}\int\limits_0^t s^{-\beta/\alpha}ds+8k L_\sigma^2C^{\star}\int\limits_0^t   \sup\limits_{y\in\mathbb{R}} {\|u_N(s,y)\|}_k^2s^{-\beta/\alpha}ds\\
			\leq & \frac{8C^{\star}\Gamma(1-\beta/\alpha)L_\sigma^2 k e^{2\gamma t}}{(2\gamma)^{1-\beta/\alpha}}\Big(1+\big[\mathcal{N}_{k,\gamma}(u_N)\big]^2\Big).
		\end{align*}
		We simplify this expression further by using the fact that $\sqrt{\ell^2+n^2}\leq |\ell|+|n|$ which is valid for all $\ell,n\in\mathbb{R}$. 
		
		It follows that
		\begin{equation}\label{I2-bound}
			I_2
			\leq  \frac{4\sqrt{C^{\star}\Gamma(1-\beta/\alpha)k}L_\sigma e^{\gamma t}}{(2\gamma)^{{\frac{1-\beta/\alpha}{2}}}}\Big[1+\mathcal{N}_{k,\gamma}(u_N)\Big].
		\end{equation}
		We then combine \eqref{I1-bound} and \eqref{I2-bound} to see that
		\begin{align*}
			{\|u_N(t,x)\|}_k\leq& {\|u_0\|}_{L^{\infty}(\mathbb{R})}+L_b\frac{e^{\gamma t}}{\gamma}\Big[1+\mathcal{N}_{k, \gamma}(u_N)\Big]+\frac{4\sqrt{C^{\star}\Gamma(1-\beta/\alpha)k}L_\sigma e^{\gamma t}}{(2\gamma)^{{\frac{1-\beta/\alpha}{2}}}}\Big[1+\mathcal{N}_{k,\gamma}(u_N)\Big]\\
			\leq& {\|u_0\|}_{L^{\infty}(\mathbb{R})}+e^{\gamma t}\Bigg[\frac{L_b}{\gamma}+\frac{4\sqrt{C^{\star}\Gamma(1-\beta/\alpha)k}L_\sigma }{(2\gamma)^{{\frac{1-\beta/\alpha}{2}}}}\Bigg]\Big(1+\mathcal{N}_{k,\gamma}(u_N)\Big).
		\end{align*}
		This estimate in turn implies that
		$$
		\mathcal{N}_{k,\gamma}(u_N)\leq {\|u_0\|}_{L^{\infty}(\mathbb{R})}+\Bigg[\frac{L_b}{\gamma}+\frac{4\sqrt{C^{\star}\Gamma(1-\beta/\alpha)k}L_\sigma }{(2\gamma)^{{\frac{1-\beta/\alpha}{2}}}}\Bigg]\Big(1+\mathcal{N}_{k,\gamma}(u_N)\Big).
		$$
		Now, set $\gamma=4\big(4\sqrt{C^{\star}\Gamma(1-\beta/\alpha)k}L_\sigma \big)^{\frac{2}{1-\beta/\alpha}}$ and choose $k\geq \max\big(2,L_b^{1-\beta/\alpha}L_\sigma^{-2}\big)$ to see that $$\frac{L_b}{\gamma}+\frac{4\sqrt{C^{\star}\Gamma(1-\beta/\alpha)k}L_\sigma }{(2\gamma)^{{\frac{1-\beta/\alpha}{2}}}}\leq 3/4.$$ Finally, solve for $\mathcal{N}_{k,\gamma}(u_N)$ to find
		$\mathcal{N}_{k,4\big(4\sqrt{C^{\star}\Gamma(1-\beta/\alpha)k}L_\sigma \big)^{\frac{2}{1-\beta/\alpha}}}(u_N)\leq 4({\|u_0\|}_{L^{\infty}(\mathbb{R})}+1)$, that is, $$\sup\limits_{N>0}\sup\limits_{x\in\mathbb{R}}\mathbb{E}\big(|u_N(t,x)|^k\big)\leq 4^k({\|u_0\|}_{L^{\infty}(\mathbb{R})}+1)^k\exp{4\big(4\sqrt{C^{\star}\Gamma(1-\beta/\alpha)}L_\sigma \big)^{\frac{2}{1-\beta/\alpha}}k^{1+\frac{1}{1-\beta/\alpha}} t}.$$
		\begin{enumerate}
			\item[2.] \underline{$\sigma: \mathbb{R}_+\times\mathbb{R}\rightarrow\mathbb{R}$ is bounded}: 
		\end{enumerate}
		In this case, we modify the proof of the case $L_\sigma>0$ by first observing that \eqref{Trc-norm},  and \eqref{I1-bound} remain valid. We use \eqref{I2-sq} combined with \eqref{HeatKernel} to estimate $I_2$ as follows:
		\begin{equation}\label{I2-sq}
			I_2^2\leq 4kC^{\star}{\|\sigma\|}_{L^\infty(\mathbb{R_{+}}\times\mathbb{R})}^2 \int\limits_0^ts^{-\beta/\alpha} ds \leq \frac{4C^{\star}k}{1-\beta/\alpha}{\|\sigma\|}_{L^\infty(\mathbb{R_{+}}\times\mathbb{R})}^2 t^{1-\beta/\alpha}.
		\end{equation}
		This yields
		\begin{align*}
			{\|u_N(t,x)\|}_k\leq& {\|u_0\|}_{L^{\infty}(\mathbb{R})}+L_b\frac{e^{\gamma t}}{\gamma}\Big[1+\mathcal{N}_{k, \gamma}(u_N)\Big]+\frac{2\sqrt{C^{\star}k}}{\sqrt{1-\beta/\alpha}}{\|\sigma\|}_{L^\infty(\mathbb{R_{+}}\times\mathbb{R})}t^{\frac{1-\beta/\alpha}{2}}.
		\end{align*}

		Since  $t^ae^{-bt}\leq\Big( \frac{a}{b}\Big)^ae^{-a}=:C(a,b)$ for all $0<a<1$ and $b>0$, we get
		
			\begin{align*}
			{\|u_N(t,x)\|}_k\leq& {\|u_0\|}_{L^{\infty}(\mathbb{R})}+L_b\frac{e^{\gamma t}}{\gamma}\Big[1+\mathcal{N}_{k, \gamma}(u_N)\Big]+\frac{2\sqrt{C^{\star}k}}{\sqrt{1-\beta/\alpha}}{\|\sigma\|}_{L^\infty(\mathbb{R_{+}}\times\mathbb{R})}\Big(\frac{1-\beta/\alpha}{2\gamma}\Big)^{\frac{1-\beta/\alpha}{2}}e^{-\frac{1-\beta/\alpha}{2}}e^{\gamma t}.
		\end{align*}

		 Now divide both sides of the preceding inequality by $e^{\gamma t}$. Since the right-hand side of the resulting inequality does not depend on $(t,x)$,  optimizing over $(t,x)$ yields
		
		\begin{align*}
			\mathcal{N}_{k,\gamma}(u_N)\leq& {\|u_0\|}_{L^{\infty}(\mathbb{R})}+\frac{L_b}{\gamma}\Big[1+\mathcal{N}_{k, \gamma}(u_N)\Big]+\frac{2\sqrt{C^{\star}k}}{\sqrt{1-\beta/\alpha}}{\|\sigma\|}_{L^\infty(\mathbb{R_{+}}\times\mathbb{R})}C_{\alpha, \beta, \gamma}.
		\end{align*}
		uniformly for all real numbers $k\geq 2, N, \gamma>0.$ Now we set $\gamma=2L_b$ and solve for $\mathcal{N}_{k, 2L_b}(u_N)$ to find 
		\begin{align*}
			\mathcal{N}_{k, 2L_b}(u_N)\leq &2{\|u_0\|}_{L^{\infty}(\mathbb{R})}+ \frac{4C_{\alpha, \beta, \gamma}\sqrt{C^{\star}k}}{\sqrt{1-\beta/\alpha}}{\|\sigma\|}_{L^\infty(\mathbb{R_{+}}\times\mathbb{R})}+1 \\
			\leq & C_{\star}\Big({\|u_0\|}_{L^{\infty}(\mathbb{R})}+ {\|\sigma\|}_{L^\infty(\mathbb{R_{+}}\times\mathbb{R})}+1\Big)\sqrt{k},
		\end{align*}
		where $C_{\star}=\max\Big(2, \frac{4C_{\alpha, \beta, \gamma}\sqrt{C^{\star}}}{\sqrt{1-\beta/\alpha}}\Big)$. This is equivalent to
		$$
		\sup\limits_{N>0}\sup\limits_{x\in\mathbb{R}}\mathbb{E}\big(|u_N(t,x)|^k\big)\leq C_{\star}^k e^{2kL_bt}\Big({\|u_0\|}_{L^\infty(\mathbb{R})}+{\|\sigma\|}_{L^\infty(\mathbb{R_{+}}\times\mathbb{R})}+1\Big)^kk^{k/2}.
		$$
		This concludes the proof.
		
	\end{proof}
	
	\begin{proposition}\label{Prob6}
		\begin{enumerate}
			\item For $L_\sigma>0,$ we have 
			$$ \mathbb{P}\big(|u_{N+1}(t,x)|\geq e^N\big)\leq\exp{-\frac{N^{2-\beta/\alpha}}{(C_{\#}L_\sigma)^2 (8t)^{1-\beta/\alpha}}}
			$$
			uniformly for all $t>0$ $x\in\mathbb{R}$ and $N\geq 4\log C_0\vee 8C_{\#}^{\frac{2}{1-\beta/\alpha}}t\max\Big(2^{\frac{1}{1-\beta/\alpha}}L_\sigma^{\frac{2}{1-\beta/\alpha}}, L_b\Big)$ and $C_0, C_{\#}$, are defined in Proposition \ref{Prob4}.
			\item If $\sigma\in L^\infty (\mathbb{R}_+\times\mathbb{R})$, then
			$$ \mathbb{P}\big(|u_{N+1}(t,x)|\geq e^N\big)\leq\exp{-\frac{e^{2N-4L_bt}}{eC_{\star}^2\big({\|u_0\|}_{L^\infty(\mathbb{R})}+{\|\sigma\|}_{L^\infty(\mathbb{R_{+}}\times\mathbb{R})}+1\big)^2}}
			$$
			uniformly for all $t>0$ $x\in\mathbb{R}$ and $$N\geq\frac{1}{2}+\log C_{\star}+2L_bt+\log\big({\|u_0\|}_{L^\infty(\mathbb{R})}+{\|\sigma\|}_{L^\infty(\mathbb{R_{+}}\times\mathbb{R})}+1\big),$$ and $ C_{\star}$ is defined in Proposition \ref{Prob4}.
		\end{enumerate} 
	\end{proposition}
	\begin{proof}
		As usual, we treat the two cases separately.
		\begin{enumerate}
			\item \underline{$L_\sigma>0$}:
		\end{enumerate}
		By combining Proposition \ref{Prob4} first part and Markov's Inequality, we get
		\begin{align*}
			\mathbb{P}\big(|u_{N+1}(t,x)|\geq& e^N\big)\leq e^{-kN}\mathbb{E}\big(|u_{N+1}|^k\big) \\ 
			\leq & C_0^k\exp{-kN+4(C_{\#}L_\sigma)^{\frac{2}{1-\beta/\alpha}}k^{1+\frac{1}{1-\beta/\alpha}} t}
		\end{align*}
		uniformly for all real numbers $N, t>0, x\in\mathbb{R}$ and $k\geq \max\big(2,L_b^{1-\beta/\alpha}L_\sigma^{-2}\big)$. Now set $k=(AN)^{1-\beta/\alpha}$--where $A>0$ is to be determined--in order to see that
		\begin{align*}
			\mathbb{P}\big(|u_{N+1}(t,x)|\geq e^N\big)\leq&  
			\exp{-\Big(1-\frac{\log C_0}{N}-4(C_{\#}L_\sigma)^{\frac{2}{1-\beta/\alpha}} At\Big)\frac{N^{2-\beta/\alpha}}{A^{\beta/\alpha-1}}}.
		\end{align*}
		Now set $A=\Big(8(C_{\#}L_\sigma)^{\frac{2}{1-\beta/\alpha}} t\Big)^{-1}$  in order to see that
		\begin{align*}
			\mathbb{P}\big(|u_{N+1}(t,x)|\geq e^N\big)\leq&  
			\exp{-\Big(\frac{1}{2}-\frac{\log C_0}{N}\Big)\frac{N^{2-\beta/\alpha}}{(C_{\#}L_\sigma)^2 (8t)^{1-\beta/\alpha}}},
		\end{align*}
		which is the desired outcome when $L_\sigma>0$ provided additionally that $N\geq 4\log C_0\vee 8C_{\#}^{\frac{2}{1-\beta/\alpha}}t\max\Big(2^{\frac{1}{1-\beta/\alpha}}L_\sigma^{\frac{2}{1-\beta/\alpha}}, L_b\Big)$.
		\begin{enumerate}
			\item[2.] \underline{$\sigma\in L^\infty(\mathbb{R}_+\times\mathbb{R})$}: 
		\end{enumerate}
		This case is proved similarly, but uses the second part of Proposition \ref{Prob4} instead. We get
		\begin{equation}\label{PN}
			\mathbb{P}\big(|u_{N+1}(t,x)|\geq e^N\big) 
			\leq  C_{\Xi}^ke^{-kN}k^{k/2},
		\end{equation}
		where $C_{\Xi}=C_{\star}e^{2L_bt}\big({\|u_0\|}_{L^\infty(\mathbb{R})}+{\|\sigma\|}_{L^\infty(\mathbb{R_{+}}\times\mathbb{R})}+1\big)$. Next, set $k=C_{\Xi}^{-2}e^{2N-1}$ in \eqref{PN} to find that
		\begin{align*}
			\mathbb{P}\big(|u_{N+1}(t,x)|\geq e^N\big) 
			\leq\exp{-\frac{e^{2N-4L_b}}{eC_{\star}^2\big({\|u_0\|}_{L^\infty(\mathbb{R})}+{\|\sigma\|}_{L^\infty(\mathbb{R_{+}}\times\mathbb{R})}+1\big)^2}}
		\end{align*}
		uniformly for $N\geq\frac{1}{2}+\log C_{\star}+2L_bt+\log\big({\|u_0\|}_{L^\infty(\mathbb{R})}+{\|\sigma\|}_{L^\infty(\mathbb{R_{+}}\times\mathbb{R})}+1\big)$ and this concludes the proof.
	\end{proof}
	\begin{proposition}[\cite{FooSanLart}, Lemma 2.5]\label{lm:2.5}
		Consider a function $f:\mathbb{R}_+\rightarrow\mathbb{R}_+$ and an increasing function   $g:\mathbb{R}_+\rightarrow\mathbb{R}_+$. If there exists $a,T_0>$ such that
		$$
		\sup\limits_{t\in(0,T]}\big[e^{-at}f(t)\big]\leq e^{-aT}g(T) \qquad \forall T\in(0,T_0),
		$$
		then $\sup\limits_{t\in(0,T]}f(t)\leq g(T) \qquad \forall T\in(0,T_0).$
	\end{proposition}

	We are now in position to prove our main result.

	\section{Proof of Theorem \ref{MainTh}}\label{Proof-Thm}
	We follow the main steps in the proof of the main Theorem in \cite{FooSanLart}. We again treat the two cases separately, although a slight modification of the argument used for the first case can lead to the proof of the second case.
	\subsection{Proof of existence}
	
	The proof of existence consists of two steps: we first show that the sequence of truncated solutions $\big\{u_N\big\}_N$ converges pointwise in $(t,x)$ in $L^2(\Omega)$ to a random field, say $u$. We then show that the random field $u$ is a mild solution of Eq \eqref{MainEq}.

	\subsubsection{Proof of existence when $L_\sigma>0$}
	
	\textbf{Step 1:} For all fixed $T>0$, we show that $
		\sum\limits_{N=1}^\infty\sup\limits_{t\in(0,T]}   \sup\limits_{x\in\mathbb{R}}{\|u_{N+1}(t,x)-u_N(t,x)\|}_{k} <\infty \qquad \forall k\geq 1.
	$
	 To this aim, we slightly modify \eqref{N-norm} and define
	\begin{equation}\label{NT-norm}
		\mathcal{N}_{k, T,\gamma}(Z)=\sup\limits_{t\in(0,T]}\sup\limits_{x\in\mathbb{R}}e^{-\gamma t} {\|Z(t,x)\|}_k
	\end{equation}
	for every $\gamma>0, k\geq 1,$ and all space-time random fields $Z.$ Note that $N\mapsto\text{Lip}_N(b)$ and $N\mapsto\text{Lip}_N(\sigma)$ are non-decreasing, and $b$ and $\sigma$ are globally Lipschitz when $0<\lim\limits_{N\rightarrow\infty}\text{Lip}_N(b),\text{Lip}_N(\sigma)<\infty.$ Thus, without loss of generality, we assume that 
	\begin{equation}\label{lim-Lip}
		\lim\limits_{N\rightarrow\infty} \text{L}_{N,\sigma}=\infty,
	\end{equation}
	where $\text{Lip}_{N,\sigma}$ is defined in \eqref{LN}.
	
	Using \eqref{Trc-MildSol}, for all $k\geq 1, t>0$ and $x\in\mathbb{R}$, we have
	\begin{equation}\label{Norm-Diff}
		{\|u_{N+1}(t,x)-u_N(t,x)\|}_k\leq I_1+I_2,
	\end{equation}
	where
	\begin{equation}\label{I1-thm}
		I_1=\int\limits_0^t\int\limits_{\mathbb{R}}G_{t-s}(y-x){\|b_{N+1}\big(s,u_{N+1}(s,y)\big)-b_{N}\big(s,u_{N}(s,y)\big)\|}_k ds\ dy  
	\end{equation}
	and
	\begin{equation}\label{I2-thm}
		I_2=\left\Vert\int\limits_0^t\int\limits_{\mathbb{R}}G_{t-s}(y-x)\big[\sigma_{N+1}\big(s,u_{N+1}(s,y)\big)-\sigma_{N}\big(s,u_{N}(s,y)\big)\big] ds\ dy  \right\Vert_k.
	\end{equation}
	For all $N,s>0$ and $y\in \mathbb{R}$, consider the event 
	\begin{equation}\label{GN}
		\mathcal{G}_{N+1}(s,y)=\{\omega\in\Omega: |u_{N+1}(s,y)|(\omega)\leq e^N\}.
	\end{equation}
	First note that
	\begin{equation}\label{b-lower}
		\begin{split}
			&\left\Vert \big[b_{N+1}\big(s,u_{N+1}(s,y)\big)-b_{N}\big(s,u_{N}(s,y)\big)\big]\mathbf{1}_{\mathcal{G}_{N+1}}(s,y)\right\Vert_k \\
			& \qquad\leq  \left\Vert b_{N+1}\big(s,u_{N+1}(s,y)\big)-b_{N}\big(s,u_{N}(s,y)\big)\right\Vert_k\\
			&\qquad\leq  \text{L}_{N,b}\left\Vert u_{N+1}(s,y)-u_{N}(s,y)\right\Vert_k \leq \text{L}_{N,b}e^{\gamma s}\mathcal{N}_{k, T,\gamma}(u_{N+1}-u_N)
		\end{split}
	\end{equation}
	for all $\gamma, N>0, s\in(0,T]$ and $y\in\mathbb{R}$.
	
	Moreover,
	
	\begin{align*}
		&\left\Vert \big[b_{N+1}\big(s,u_{N+1}(s,y)\big)-b_{N}\big(s,u_{N}(s,y)\big)\big]\mathbf{1}_{\Omega\setminus\mathcal{G}_{N+1}}(s,y)\right\Vert_k \\
		& \qquad\leq  \left\Vert b_{N+1}\big(s,u_{N+1}(s,y)\big)\mathbf{1}_{\Omega\setminus\mathcal{G}_{N+1}}(s,y)\right\Vert_k+\left\Vert b_{N}\big(s,u_{N}(s,y)\big)\mathbf{1}_{\Omega\setminus\mathcal{G}_{N+1}}(s,y)\right\Vert_k\\
		&\qquad\leq \Big[\left\Vert b_{N+1}\big(s,u_{N+1}(s,y)\big)\right\Vert_{2k}+ \left\Vert b_{N}\big(s,u_{N}(s,y)\big)\right\Vert_{2k} \Big]\Big[1-\mathbb{P}\big( \mathcal{G}_{N+1}(s,y)\big)\Big]^{\frac{1}{2k}},
	\end{align*}
	where we have used, in the last line, the following  variation of the Cauchy-Schwarz inequality: $\left\Vert X\mathbf{1}_F\right\vert_k\leq \left\Vert X\right\Vert_{2k}[\mathbb{P}(F)]^{\frac{1}{2k}}$ for all $X\in L^k(\Omega)$ and $F\subset\Omega.$
	
	Now, using the first part of Proposition \ref{Prob4}, it follows that
	\begin{align*}
		&\left\Vert b_{N+1}\big(s,u_{N+1}(s,y)\big)\right\Vert_{2k}+ \left\Vert b_{N}\big(s,u_{N}(s,y)\big)\right\Vert_{2k}\\
		&\qquad\leq  L_b\Big[\left\Vert u_{N+1}(s,y)\right\Vert_{2k}+ \left\Vert u_{N}(s,y)\big)\right\Vert_{2k} \Big] \leq 2C_0L_b\exp{4(C_{\#}L_\sigma)^{\frac{2}{1-\beta/\alpha}}k^{\frac{1}{1-\beta/\alpha}} s}
	\end{align*}
	uniformly for all $N,s>0, y\in\mathbb{R}$ and $k\geq\max\big(1,\frac{1}{2}L_b^{1-\beta/\alpha}L_{\sigma}^{-2}\big)$. It follows that
	\begin{align*}
		&\left\Vert \big[b_{N+1}\big(s,u_{N+1}(s,y)\big)-b_{N}\big(s,u_{N}(s,y)\big)\big]\mathbf{1}_{\Omega\setminus\mathcal{G}_{N+1}}(s,y)\right\Vert_k \\
		& \qquad\leq 2C_0L_b\exp{4(C_{\#}L_\sigma)^{\frac{2}{1-\beta/\alpha}}k^{\frac{1}{1-\beta/\alpha}} s}\Big[\mathbb{P}\big( |u_{N+1}(s,y)|\geq e^N\big)\Big]^{\frac{1}{2k}},
	\end{align*}
	valid uniformly for all $N,s>0, y\in\mathbb{R}$ and $k\geq\max\big(1,\frac{1}{2}L_b^{1-\beta/\alpha}L_{\sigma}^{-2}\big)$. Therefore, Proposition \ref{Prob6}--part 1 yields
	\begin{equation}\label{bN-part2}
		\begin{split}
			&\left\Vert \big[b_{N+1}\big(s,u_{N+1}(s,y)\big)-b_{N}\big(s,u_{N}(s,y)\big)\big]\mathbf{1}_{\Omega\setminus\mathcal{G}_{N+1}}(s,y)\right\Vert_k \\
			& \qquad\leq 2C_0L_b\exp{4(C_{\#}L_\sigma)^{\frac{2}{1-\beta/\alpha}}k^{\frac{1}{1-\beta/\alpha}} s}\exp{-\frac{N^{2-\beta/\alpha}}{2k(C_{\#}L_\sigma)^2 (8s)^{1-\beta/\alpha}}},
		\end{split}
	\end{equation}
	valid uniformly for all $s\in(0,T]$, $y\in\mathbb{R}$, $N\geq 4\log C_0\vee 8C_{\#}^{\frac{2}{1-\beta/\alpha}}t\max\Big(2^{\frac{1}{1-\beta/\alpha}}L_\sigma^{\frac{2}{1-\beta/\alpha}}, L_b\Big)$ and $k\geq\max\big(1,\frac{1}{2}L_b^{1-\beta/\alpha}L_{\sigma}^{-2}\big)$. Thus we find that
	\begin{equation}\label{I1-3.3}
		\begin{split}
			I_1\leq& \text{L}_{N,b}e^{\gamma t}\mathcal{N}_{k,\gamma,T}(u_{N+1}-u_N)\int\limits_0^t\int\limits_{\mathbb{R}}e^{-\gamma(t-s)}G_{t-s}(y-x)\ dy \ ds\\
			&\qquad+ 2C_0L_b\int\limits_0^t\int\limits_{\mathbb{R}}\exp{4(C_{\#}L_\sigma)^{\frac{2}{1-\beta/\alpha}}k^{\frac{1}{1-\beta/\alpha}} s}\exp{-\frac{N^{2-\beta/\alpha}}{2k(C_{\#}L_\sigma)^2 (8s)^{1-\beta/\alpha}}} G_{t-s}(y-x)\ dyds\\
			\leq & \frac{\text{L}_{N,b}e^{\gamma t}}{\gamma}\mathcal{N}_{k,\gamma,T}(u_{N+1}-u_N)\\
			& \qquad+\frac{2C_0L_b}{4(C_{\#}L_\sigma)^{\frac{2}{1-\beta/\alpha}}k^{\frac{1}{1-\beta/\alpha}}}\exp{4(C_{\#}L_\sigma)^{\frac{2}{1-\beta/\alpha}}k^{\frac{1}{1-\beta/\alpha}} t-\frac{N^{2-\beta/\alpha}}{2k(C_{\#}L_\sigma)^2 (8t)^{1-\beta/\alpha}}} 
		\end{split}
	\end{equation}
	for all $t\in(0,T]$ and $x\in\mathbb{R}$, provided that $N,k\geq c$ where 
	\begin{equation}\label{c}
		c=c\big(\alpha, \beta, T, L_b,L_\sigma, {\| u_0\|}_{L^{\infty}(\mathbb{R})}\big)>1 
	\end{equation}
	is a constant whose value is unimportant to this analysis. It is worth mentioning that the constant $c,$ while fixed, can be chosen to be as large as we wish. Keeping in mind \eqref{Ass3-1}, we therefore select $c\big(\alpha, \beta, T, L_b,L_\sigma, {\| u_0\|}_{L^{\infty}(\mathbb{R})}\big)$ large enough to also ensure that
	
	\begin{equation}\label{c-sup}
		c>\sup\limits_{N\geq N_0}\frac{ \text{L}_{N,b}^{1-\beta/\alpha}}{ \text{L}_{N,\sigma}^{2}}, \ \text{where} \ N_0=\inf{\{N>1: \text{L}_{N,\sigma}>1 \}}.
	\end{equation}
	Clearly $N_0<\infty$  thanks to \eqref{lim-Lip}. 
	
	We now estimate $I_2$. To this aim, we apply the Burkh\"older-Davis-Gundy inequality (see for example \cite{DaViMuLarXiao}) to find that
	\begin{equation}\label{I2-BDG}
		I_2^2\leq\int\limits_0^t\int\limits_{\mathbb{R}}\Big[G_{t-s}(y-x)\big]^2\left\Vert\sigma_{N+1}\big(s,u_{N+1}(s,y)\big)-\sigma_{N}\big(s,u_{N}(s,y)\big)\right\Vert_k^2 dy\ ds  .
	\end{equation}
	As with the function $b_N$, we also have
	\begin{equation}\label{sig-lower}
		\begin{split}
			&\left\Vert \big[\sigma_{N+1}\big(s,u_{N+1}(s,y)\big)-\sigma_{N}\big(s,u_{N}(s,y)\big)\big]\mathbf{1}_{\mathcal{G}_{N+1}}(s,y)\right\Vert_k \\
			&\qquad\leq \text{L}_{N,\sigma}e^{\gamma s}\mathcal{N}_{k, T,\gamma}(u_{N+1}-u_N)
		\end{split}
	\end{equation}
	for all $\gamma, N>0, s\in(0,T]$ and $y\in\mathbb{R}$. Moreover,
	
	\begin{align*}
		&\left\Vert \big[\sigma_{N+1}\big(s,u_{N+1}(s,y)\big)-\sigma_{N}\big(s,u_{N}(s,y)\big)\big]\mathbf{1}_{\Omega\setminus\mathcal{G}_{N+1}}(s,y)\right\Vert_k \\
		& \qquad\leq 2C_0L_\sigma\exp{4(C_{\#}L_\sigma)^{\frac{2}{1-\beta/\alpha}}k^{\frac{1}{1-\beta/\alpha}} s}\exp{-\frac{N^{2-\beta/\alpha}}{2k(C_{\#}L_\sigma)^2 (8s)^{1-\beta/\alpha}}},
	\end{align*}
	uniformly for all $s\in(0,T]$, $y\in\mathbb{R}$, $k\geq 1$ and 
	\begin{equation}\label{cT}
		N\geq 4\log C_0\vee 8C_{\#}^{\frac{2}{1-\beta/\alpha}}T\max\Big(2^{\frac{1}{1-\beta/\alpha}}L_\sigma^{\frac{2}{1-\beta/\alpha}}, L_b\Big):=c_T. 
	\end{equation}
	It follows that
	\begin{align*}
		I_2^2\leq&8k \text{L}_{N,\sigma}^2e^{2\gamma t}\big[\mathcal{N}_{k, T,\gamma}(u_{N+1}-u_N)\big]^2 \int\limits_0^t\int\limits_{\mathbb{R}}e^{-2\gamma(t-s)}\big[G_{t-s}(y-x)\big]^2dy\ ds\\
		&\qquad+ 16kC_0^2L_\sigma^2\int\limits_0^t\int\limits_{\mathbb{R}}\exp{4(C_{\#}L_\sigma)^{\frac{2}{1-\beta/\alpha}}k^{\frac{1}{1-\beta/\alpha}} s}\exp{-\frac{N^{2-\beta/\alpha}}{2k(C_{\#}L_\sigma)^2 (8s)^{1-\beta/\alpha}}}\big[G_{t-s}(y-x)\big]^2dy\ ds\\
		\leq& \frac{8kC^{\star} \text{L}_{N,\sigma}^2e^{2\gamma t}\Gamma(1-\beta/\alpha)}{(2\gamma)^{1-\beta/\alpha}}\big[\mathcal{N}_{k, T,\gamma}(u_{N+1}-u_N)\big]^2\\
		&\qquad +\frac{16kC^{\star}C_0^2L_\sigma^2}{1-\beta/\alpha}\exp{4(C_{\#}L_\sigma)^{\frac{2}{1-\beta/\alpha}}k^{\frac{1}{1-\beta/\alpha}} t}\exp{-\frac{N^{2-\beta/\alpha}}{k(C_{\#}L_\sigma)^2 (8t)^{1-\beta/\alpha}}}t^{1-\beta/\alpha}
	\end{align*}
	for all $\gamma>0, t\in(0,T]$ and $x\in\mathbb{R}$, provided that $N\geq \max(c,c_T)$ and $k\geq c$, where the constants $c$ and $c_T$ are defined in \eqref{c} and \eqref{cT}, respectively.  We now combine \eqref{Norm-Diff} and \eqref{I1-3.3} to find
	\begin{equation}\label{norm-Cauchy}
		\begin{split}
			&{\|u_{N+1}(t,x)-u_N(t,x)\|}_k\leq e^{\gamma t}\Big[\frac{\text{L}_{N,b}}{\gamma}+\frac{2^{1+\beta/(2\alpha)}\sqrt{kC^{\star}\Gamma(1-\beta/\alpha)}\text{L}_{N,\sigma}}{\gamma^{\frac{1-\beta/\alpha}{2}}}\Big]\mathcal{N}_{k,\gamma,T}(u_{N+1}-u_N)\\
			& \qquad+\frac{2C_0L_b}{4(C_{\#}L_\sigma)^{\frac{2}{1-\beta/\alpha}}k^{\frac{1}{1-\beta/\alpha}}}\exp{4(C_{\#}L_\sigma)^{\frac{2}{1-\beta/\alpha}}k^{\frac{1}{1-\beta/\alpha}} t-\frac{N^{2-\beta/\alpha}}{2k(C_{\#}L_\sigma)^2 (8t)^{1-\beta/\alpha}}}\\
			&\qquad+ \frac{4\sqrt{kC^{\star}}C_0L_\sigma}{\sqrt{1-\beta/\alpha}}\exp{4(C_{\#}L_\sigma)^{\frac{2}{1-\beta/\alpha}}k^{\frac{1}{1-\beta/\alpha}} t}\exp{-\frac{N^{2-\beta/\alpha}}{2k(C_{\#}L_\sigma)^2 (8t)^{1-\beta/\alpha}}}t^{\frac{1-\beta/\alpha}{2}}
		\end{split}
	\end{equation}
	as long as $N\geq\max(c,c_T)$ and $k\geq c$. Next, set 
	\begin{equation}\label{k-and-g}
		k=c \qquad  \text{and} \qquad  \gamma=16\big[C^{\star}\Gamma(1-\beta/\alpha)\big]^{\frac{1}{1-\beta/\alpha}}A_0^{\frac{2}{1-\beta/\alpha}}k^{\frac{1}{1-\beta/\alpha}} \text{L}_{N,\sigma}^{\frac{2}{1-\beta/\alpha}},
	\end{equation}
	where 
	\begin{equation}\label{A0}
		A_0:=\max\Bigg(4, \frac{C_{\#}L_\sigma}{\sqrt{C^{\star}\Gamma(1-\beta/\alpha)}}, \big[C^{\star}\Gamma(1-\beta/\alpha)\big]^{-2}\Bigg).
	\end{equation}
	Note that $\gamma$ depends on $N$ and $T.$ Therefore, we find that

	\begin{align*}
		\sup\limits_{x\in\mathbb{R}}{\|u_{N+1}(t,x)-u_N(t,x)\|}_c&\leq e^{\gamma t}\Bigg[\frac{\text{L}_{N,b}}{16\big[C^{\star}\Gamma(1-\beta/\alpha)\big]^{\frac{1}{1-\beta/\alpha}}A_0^{\frac{2}{1-\beta/\alpha}}c^{\frac{1}{1-\beta/\alpha}} \text{L}_{N,\sigma}^{\frac{2}{1-\beta/\alpha}}}+\frac{1}{A_0}\Bigg]\mathcal{N}_{c,\gamma,T}(u_{N+1}-u_N)\\
		& +C_0\Bigg[\frac{2L_b}{4(C_{\#}L_\sigma)^{\frac{2}{1-\beta/\alpha}}c^{\frac{1}{1-\beta/\alpha}}}
		+ 4\sqrt{\frac{cC^{\star}}{1-\beta/\alpha}}L_\sigma t^{\frac{1-\beta/\alpha}{2}}\Bigg]\\
		&\qquad\qquad\times\exp{4(C_{\#}L_\sigma)^{\frac{2}{1-\beta/\alpha}}c^{\frac{1}{1-\beta/\alpha}} t-\frac{N^{2-\beta/\alpha}}{2c(C_{\#}L_\sigma)^2 (8t)^{1-\beta/\alpha}}}
	\end{align*}
	uniformly for all $T>0, t\in(0,T],$ and $N\geq \max(N_0,c,c_T).$ 
	 
	 
	We now apply \eqref{c-sup} to get
	
	\begin{align*}
		\sup\limits_{x\in\mathbb{R}}{\|u_{N+1}(t,x)-u_N(t,x)\|}_c&\leq e^{\gamma t}\Bigg[\frac{1}{16\big[C^{\star}\Gamma(1-\beta/\alpha)\big]^{\frac{1}{1-\beta/\alpha}}A_0^{\frac{2}{1-\beta/\alpha}}}+\frac{1}{A_0}\Bigg]\mathcal{N}_{c,\gamma,T}(u_{N+1}-u_N)\\
		& \qquad\qquad+C_0\Bigg[\frac{2L_b}{4(C_{\#}L_\sigma)^{\frac{2}{1-\beta/\alpha}}c^{\frac{1}{1-\beta/\alpha}}}
		+ 4\sqrt{\frac{cC^{\star}}{1-\beta/\alpha}}L_\sigma t^{\frac{1-\beta/\alpha}{2}}\Bigg]\\
		&\qquad\qquad\times\exp{4(C_{\#}L_\sigma)^{\frac{2}{1-\beta/\alpha}}c^{\frac{1}{1-\beta/\alpha}} t-\frac{N^{2-\beta/\alpha}}{2c(C_{\#}L_\sigma)^2 (8t)^{1-\beta/\alpha}}}\\
		&\leq e^{\gamma t}\Bigg[\frac{1}{16}+\frac{1}{4}\Bigg]\mathcal{N}_{c,\gamma,T}(u_{N+1}-u_N)\\
		& \qquad\qquad+C_0\Bigg[\frac{2L_b}{4(C_{\#}L_\sigma)^{\frac{2}{1-\beta/\alpha}}c^{\frac{1}{1-\beta/\alpha}}}
		+ 4\sqrt{\frac{cC^{\star}}{1-\beta/\alpha}}L_\sigma t^{\frac{1-\beta/\alpha}{2}}\Bigg]\\
		&\qquad\qquad\times\exp{4(C_{\#}L_\sigma)^{\frac{2}{1-\beta/\alpha}}c^{\frac{1}{1-\beta/\alpha}} t-\frac{N^{2-\beta/\alpha}}{2c(C_{\#}L_\sigma)^2 (8t)^{1-\beta/\alpha}}}\\
	\end{align*}
	uniformly for all $T>0, t\in(0,T],$ and $N\geq \max(N_0,c,c_T);$ Since $1/16+1/4<1/2,$ we proceed to divide both sides of the preceding by $e^{\gamma t}$ and optimize over $t\in(0,T]$ in order to find that
	\begin{equation}\label{N-bound}
		\begin{split}
			\mathcal{N}_{c,\gamma,T}(u_{N+1}-u_N)\leq&
			2C_0\Bigg[\frac{2L_b}{4(C_{\#}L_\sigma)^{\frac{2}{1-\beta/\alpha}}c^{\frac{1}{1-\beta/\alpha}}}
			+ 4\sqrt{\frac{cC^{\star}}{1-\beta/\alpha}}L_\sigma t^{\frac{1-\beta/\alpha}{2}}\Bigg]\\
			&\qquad\qquad\qquad\times\sup\limits_{t\in(0,T]}\exp{-\big[\gamma-4(C_{\#}L_\sigma)^{\frac{2}{1-\beta/\alpha}}c^{\frac{1}{1-\beta/\alpha}}\big] t-\frac{N^{2-\beta/\alpha}}{2c(C_{\#}L_\sigma)^2 (8t)^{1-\beta/\alpha}}},
		\end{split}
	\end{equation}
	uniformly for all $T>0$, $ t\in(0,T],$ and $N\geq \max(N_0,c,c_T)$. Combining \eqref{c-sup}, \eqref{k-and-g}, and \eqref{A0}, we get
	\begin{align*}
		\gamma-4(C_{\#}L_\sigma)^{\frac{2}{1-\beta/\alpha}}c^{\frac{1}{1-\beta/\alpha}}&\geq c^{\frac{1}{1-\beta/\alpha}}\Big(16\big[C^{\star}\Gamma(1-\beta/\alpha)\big]^{\frac{1}{1-\beta/\alpha}}A_0^{\frac{2}{1-\beta/\alpha}} \text{L}_{N,\sigma}^{\frac{2}{1-\beta/\alpha}}-4(C_{\#}L_\sigma)^{\frac{2}{1-\beta/\alpha}}\Big) \\
		&\geq c^{\frac{1}{1-\beta/\alpha}}\Big(16\big[C^{\star}\Gamma(1-\beta/\alpha)\big]^{\frac{1}{1-\beta/\alpha}}A_0^{\frac{2}{1-\beta/\alpha}} -4(C_{\#}L_\sigma)^{\frac{2}{1-\beta/\alpha}}\Big)>0,
	\end{align*}
	uniformly for $N\geq N_0.$
	
	We now consider the function 
	$$
	\psi(t)=\big[\gamma-4(C_{\#}L_\sigma)^{\frac{2}{1-\beta/\alpha}}c^{\frac{1}{1-\beta/\alpha}}\big] t+\frac{N^{2-\beta/\alpha}}{2c(C_{\#}L_\sigma)^2 (8t)^{1-\beta/\alpha}} \quad \text{for} \ t>0.
	$$
	It is not hard to show that 
	
	\begin{align*}
		\psi'(t)&\leq \gamma-\frac{1-\beta/\alpha}{2^{4-3\beta/\alpha}c(C_{\#}L_\sigma)^2 }(N/T)^{2-\beta/\alpha}\\
		&=16\big[C^{*}\Gamma(1-\beta/\alpha)\big]^{\frac{1}{1-\beta/\alpha}}A_0^{\frac{2}{1-\beta/\alpha}}c^{\frac{1}{1-\beta/\alpha}} \text{L}_{N,\sigma}^{\frac{2}{1-\beta/\alpha}}-\frac{1-\beta/\alpha}{2^{4-3\beta/\alpha}c(C_{\#}L_\sigma)^2 }(N/T)^{2-\beta/\alpha} \qquad \forall 0<t\leq T. 
	\end{align*}

	It follows that $\psi'<0$ everywhere on $(0,T]$ provided that
	\begin{equation}\label{NT}
		\frac{N}{ \text{L}_{N,\sigma}^{\frac{2}{(1-\beta/\alpha)(2-\beta/\alpha)}}}>\frac{c^{\frac{1}{1-\beta/\alpha}}(C_{\#}L_\sigma)^{\frac{2}{2-\beta/\alpha}} \big[C^{\star}\Gamma(1-\beta/\alpha)A_0^2\big]^{\frac{1}{(1-\beta/\alpha)(2-\beta/\alpha)}}T}{(1-\beta/\alpha)^{\frac{1}{2-\beta/\alpha}}}=:N_T.
	\end{equation}
	
Clearly the left-hand side of  \eqref{NT} is well-defined--for example when $N_T>N_0.$ By \eqref{Ass3-1}, we have $\frac{N}{ \text{L}_{N,\sigma}^{\frac{2}{(1-\beta/\alpha)(2-\beta/\alpha)}}}\rightarrow\infty$  as $N\rightarrow\infty.$ Therefore, \eqref{NT} holds for every $N>\max(N_0,N_T).$ It follows that
$\inf\limits_{t\in(0,T]}\psi(t)=\psi(T) \quad \text{for all} \  N>\max(N_0,N_T), \ T>0.$
	
Therefore, it follows from \eqref{N-bound} that, for every $T_0>0$ and $N\geq \max(N_0,N_{T_0},c_0,c_{T_0}),$
	
	\begin{align*}
		\begin{split}
			\mathcal{N}_{c,\gamma,T}(u_{N+1}-u_N)\leq&
			2C_0\Bigg[\frac{2L_b}{4(C_{\#}L_\sigma)^{\frac{2}{1-\beta/\alpha}}c^{\frac{1}{1-\beta/\alpha}}}
			+ 4\sqrt{\frac{cC^{\star}}{1-\beta/\alpha}}L_\sigma T^{\frac{1-\beta/\alpha}{2}}\Bigg]\\
			&\qquad\qquad\qquad\times\exp{-\big[\gamma-4(C_{\#}L_\sigma)^{\frac{2}{1-\beta/\alpha}}c^{\frac{1}{1-\beta/\alpha}}\big] T-\frac{N^{2-\beta/\alpha}}{2c(C_{\#}L_\sigma)^2 (8T)^{1-\beta/\alpha}}},
		\end{split}
	\end{align*}
	uniformly for all $T\in(0,T_0).$ We have also used the fact that $T\mapsto c_T$ is increasing thanks to \eqref{cT}. We now apply Proposition \ref{lm:2.5} in order to deduce from the above and \eqref{NT-norm} that, for every $T>0$ fixed,
	\begin{equation}\label{limSup}
		\limsup\limits_{N\rightarrow\infty}N^{-(2-\beta/\alpha)}\log\sup\limits_{t\in(0,T]}\sup\limits_{x\in\mathbb{R}}{\|u_{N+1}(t,x)-u_N(t,x)\|}_{k}<0 \qquad\forall k\geq 1.
	\end{equation}
	We note here that \eqref{limSup} clearly holds when $k\in[1,c]$. If $k>c$, then we can relabel $k$ as $c$ and vice-versa, thanks to the fact we can choose $c$ to be as large as we wish.
	This proves that 
	\begin{equation}\label{sum-Lk}
		\sum\limits_{N=1}^\infty\sup\limits_{t\in(0,T]}   \sup\limits_{x\in\mathbb{R}}{\|u_{N+1}(t,x)-u_N(t,x)\|}_{k} <\infty \qquad \forall k\geq 1,
	\end{equation} and Step 1 is proved.
	\eqref{sum-Lk} also proves that 
	\begin{equation}\label{lim-uN}
		u(t,x)=\lim\limits_{N\rightarrow\infty}u_N(t,x) \ \text{exists in } \ L^{k}(\Omega),
	\end{equation}
	and the rate of convergence does not depend on $t\in(0,T]$ nor $x\in\mathbb{R}.$ As a direct consequence, $u$ is $L^k-$continuous. This ensures that $u$ has predictable version.

	\textbf{Step 2:} We show that $u$ is a mild solution of Eq. \eqref{MainEq}.

	With this in mind, for all $t\in(0,T]$ and $x\in\mathbb{R},$ we show that 
	
	\begin{equation}\label{lim-IN}
		\lim\limits_{N\rightarrow\infty}\mathcal{I}_{b_N}^N(t,x)=\mathcal{I}_b(t,x) \ \text{and} \  \lim\limits_{N\rightarrow\infty}\mathcal{I}_{\sigma_N}^N(t,x)=\mathcal{I}_\sigma(t,x), 
	\end{equation}
	where both limits hold in $L^2(\Omega)$ and the random fields $\mathcal{I}_b,\mathcal{I}_\sigma,\mathcal{I}_b^N \ \text{and} \ \mathcal{I}_\sigma^N$ were defined in \eqref{Ib}, \eqref{Is}, \eqref{IbN} and \eqref{IsN}.  Then \eqref{Trc-MildSol} and \eqref{lim-uN} will ensure that $u$  is a mild solution to \eqref{MainEq}.
	
	Since 
	\begin{equation}\label{b-bN}
		|b_N(t,z)-b(t,z)|\leq \big(|b_N(t,z)|+|b(t,z)|\big)\mathbf{1}_{\{|z|>e^N\}}\leq 2L_b(1+|z|)\mathbf{1}_{\{|z|>e^N\}}(z) \ \text{for all} \ z\in\mathbb{R} \ \text{and}\ t,N>0,
	\end{equation}
it follows that
	\begin{align*}
		&\left\Vert\mathcal{I}_{b_N}^N(t,x) - \int\limits_0^t \int\limits_{\mathbb{R}} \ G_{t-s}(y-x)b\big(s, u_N(s,y)\big)\  dy \ ds\right\Vert_2\\
		&\qquad\qquad\leq\int\limits_0^t \int\limits_{\mathbb{R}} \ G_{t-s}(y-x) \left\Vert b_N\big(s, u_N(s,y)\big)-b\big(s, u_N(s,y)\big)\right\Vert_2 \ dy \ ds\\
		&\qquad\qquad\leq 2L_b\int\limits_0^t \int\limits_{\mathbb{R}} \ G_{t-s}(y-x)\sqrt{\mathbb{E}\big(1+|u_N(s,y)|^2; |u_N(s,y)|>e^N\big)}\  dy \ ds.
	\end{align*}
	Recall that if $X\geq 0$ is a random variable and $A>0,$ combining the Cauchy-Schwarz and Markov inequalities, one can show that $\mathbb{E}(X^2; X>A)\leq \sqrt{\mathbb{E}(X^4)\mathbb{P}(|X|>A)}\leq A^{-2}\mathbb{E}(X^4).$ Thus, using Proposition \ref{Prob4}-part 1 and Proposition \ref{Prob6}-part 1 we get
	\begin{equation}\label{lim-Tail}
		\lim\limits_{N\rightarrow\infty} \sup\limits_{t\in(0,T]}\sup\limits_{y\in\mathbb{R}}\mathbb{E}\big(1+|u_N(s,y)|^2; |u_N(s,y)|>e^N\big)=0,
	\end{equation}
	and hence 
	\begin{equation}\label{lim-difIbN}
		\mathcal{I}_{b_N}^N(t,x) - \int\limits_0^t \int\limits_{\mathbb{R}} \ G_{t-s}(y-x)b\big(s, u_N(s,y)\big)\  dy \ ds\rightarrow0 \ \text{in} \ L^2(\Omega)\ \text{as} \ N\rightarrow\infty.
	\end{equation}
	
	Because the function $b$ has at-most linear growth, using Minkowski inequality, we have that $\left\Vert b\big(s, u_N(s,y)\big)-b\big(s, u(s,y)\big)\right\Vert_2$ is bounded uniformly in $s\in(0,T]$, $y\in\mathbb{R}$ and $ N>0$  thanks to Proposition \ref{Prob4}--part 1. Next, using \eqref{lim-uN}, uniform integrability and the continuity of $b$, it follows that 
	
	$$
	\lim\limits_{N\rightarrow\infty}b\big(s, u_N(s,y)\big)=b\big(s, u(s,y)\big) \ \text{in} \ L^2(\Omega), \ \text{for all} \ s>0 \ \text{and} \  y\in\mathbb{R}.
	$$
	Then, by the Dominated Convergence Theorem, we get
	$$
	\lim\limits_{N\rightarrow\infty}\int\limits_0^t \int\limits_{\mathbb{R}} \ G_{t-s}(y-x) \left\Vert b_N\big(s, u_N(s,y)\big)-b\big(s, u_N(s,y)\big)\right\Vert_2 \ dy ds=0, 
	$$
	 for all $t\in(0,T]$ and $x\in\mathbb{R}$. Then an application of the triangle inequality shows that
	\begin{equation}\label{lim-Ib}
		\lim\limits_{N\rightarrow\infty}\int\limits_0^t \int\limits_{\mathbb{R}} \ G_{t-s}(y-x)b\big(s, u_N(s,y)\big)\  dy ds=\mathcal{I}_b(t,x) \qquad \forall t\in(0,T] \qquad \text{and} \ x\in\mathbb{R}.
	\end{equation}
	Thus the first assertion of \eqref{lim-IN} is proved.

	For the second assertion of  \eqref{lim-IN}, we apply similar analyses to the function $\sigma$ and its truncated part $\sigma_N$. Similarly to \eqref{b-bN}, we have
	\begin{align*}
		|\sigma_N(t,z)-\sigma(t,z)|\leq 2L_\sigma(1+|z|)\mathbf{1}_{\{|z|>e^N\}}(z) \ \text{for all} \ z\in\mathbb{R} \ \text{and}\ t,N>0.
	\end{align*}
	Then, the $L^2(\Omega)-$isometry of stochastic integrals yields
	\begin{align*}
		&\left\Vert\mathcal{I}_{\sigma_N}^N(t,x) - \int\limits_0^t \int\limits_{\mathbb{R}} \ G_{t-s}(y-x)\sigma\big(s, u_N(s,y)\big)W( dy ds)\right\Vert_2^2\\
		&\qquad\qquad\leq\int\limits_0^t \int\limits_{\mathbb{R}} \big[G_{t-s}(y-x)\big]^2 \left\Vert \sigma_N\big(s, u_N(s,y)\big)-\sigma\big(s, u_N(s,y)\big)\right\Vert_2^2 \ dy  ds \\
		&\qquad\qquad\leq 4L_\sigma^2\int\limits_0^t \int\limits_{\mathbb{R}} \big[G_{t-s}(y-x)\big]^2 \mathbb{E}\big(1+|u_N(s,y)|^2; |u_N(s,y)|>e^N\big)\ dy ds.
	\end{align*}
Again, Proposition \ref{Prob4}-part 1 and Proposition \ref{Prob6}-part 1 yield
	\begin{align*}
		\lim\limits_{N\rightarrow\infty} \sup\limits_{s\in(0,T]}\sup\limits_{y\in\mathbb{R}}\mathbb{E}\big(1+|u_N(s,y)|^2; |u_N(s,y)|>e^N\big)=0,
	\end{align*}
	
	and hence
	\begin{equation}\label{I-sg-N}
		\mathcal{I}_{\sigma_N}^N(t,x) - \int\limits_0^t \int\limits_{\mathbb{R}} \ G_{t-s}(y-x)\sigma\big(s, u_N(s,y)\big)W( dy ds)\rightarrow0 \ \text{in} \ L^2(\Omega) \ \text{as} \ N\rightarrow\infty.
	\end{equation}

The function $\sigma$ also has  at-most linear growth, so using Minkowski inequality, it follows that $\left\Vert \sigma\big(s, u_N(s,y)\big)-\sigma\big(s, u(s,y)\big)\right\Vert_2$ is bounded uniformly in $s\in(0,T]$, $y\in\mathbb{R}$ and $ N>0$  thanks to Proposition \ref{Prob4}--part 1. Next, using \eqref{lim-uN}, uniform integrability and the continuity of $\sigma$, we get
	
	$$
	\lim\limits_{N\rightarrow\infty}\sigma\big(s, u_N(s,y)\big)=\sigma\big(s, u(s,y)\big) \ \text{in} \ L^2(\Omega), \ \text{for all} \ s>0 \ \text{and} \  y\in\mathbb{R}.
	$$
	Then, by the Dominated Convergence Theorem, we have
	
		\begin{equation}\label{last-Int}
			\lim\limits_{N\rightarrow\infty}\left\Vert\mathcal{I}_\sigma(t,x) - \int\limits_0^t \int\limits_{\mathbb{R}} \ G_{t-s}(y-x)\sigma\big(s, u_N(s,y)\big)W( ds dy)\right\Vert_2^2\rightarrow0  
	\end{equation}

	for all $t\in(0,T]$ and $x\in\mathbb{R}$. Finally, an application of the triangle inequality shows that
	\begin{equation}\label{Int-eq-I-sg}
		\lim\limits_{N\rightarrow\infty}  \int\limits_0^t \int\limits_{\mathbb{R}} \ G_{t-s}(y-x)\sigma\big(s, u_N(s,y)\big)W( dy ds)= \mathcal{I}_\sigma(t,x),
	\end{equation}
	
	and the second assertion of \eqref{lim-IN} is proved. This concludes the proof.
	
	\subsubsection{Proof of existence when $\sigma\in L^\infty(\mathbb{R}_+\times\mathbb{R})$}
	The proof is very similar to the previous case following the same two steps with minor modifications. We provide the outline below. In Step 1,  first note that \eqref{Norm-Diff}, \eqref{I1-thm}
	and \eqref{I2-thm} remain valid. The same is true for \eqref{b-lower}. As for \eqref{bN-part2}, using  the second part of Proposition \ref{Prob6} instead leads to

	\begin{equation}\label{bN-part2bis}
		\begin{split}
			&\left\Vert \big[b_{N+1}\big(s,u_{N+1}(s,y)\big)-b_{N}\big(s,u_{N}(s,y)\big)\big]\mathbf{1}_{\Omega\setminus\mathcal{G}_{N+1}}(s,y)\right\Vert_k \\
			& \quad\leq 2L_be^{2L_bs}C_{\star}\sqrt{k}\Big({\|u_0\|}_{L^\infty(\mathbb{R})}+{\|\sigma\|}_{L^\infty(\mathbb{R_{+}}\times\mathbb{R})}+1\Big)\exp{-\frac{e^{2N-4L_bs}}{(2k)eC_{\star}^2\big({\|u_0\|}_{L^\infty(\mathbb{R})}+{\|\sigma\|}_{L^\infty(\mathbb{R_{+}}\times\mathbb{R})}+1\big)^2}},
		\end{split}
	\end{equation}
	which is valid uniformly for all $s\in(0,T]$, $y\in\mathbb{R}$, $N\geq 4\log C_0\vee C_{\#}^{\frac{2}{1-\beta/\alpha}}t\max\Big(2^{\frac{1}{1-\beta/\alpha}}L_\sigma^{\frac{2}{1-\beta/\alpha}}, L_b\Big)$ and $k\geq\max\big(1,\frac{1}{2}L_b^{1-\beta/\alpha}L_{\sigma}^{-2}\big)$.
	It follows that
	\begin{equation}\label{I1-SgbDd}
		\begin{split}
			I_1
			\leq & \text{L}_{N,b}\frac{e^{\gamma t}}{\gamma}\mathcal{N}_{k,\gamma,T}(u_{N+1}-u_N)
			+ 2L_be^{2L_bt}C_{\star}\sqrt{k}t\Big({\|u_0\|}_{L^\infty(\mathbb{R})}+{\|\sigma\|}_{L^\infty(\mathbb{R_{+}}\times\mathbb{R})}+1\Big)\\
			&\hspace{5.5cm}\times \exp{-\frac{e^{2N-4L_bt}}{(2k)eC_{\star}^2\big({\|u_0\|}_{L^\infty(\mathbb{R})}+{\|\sigma\|}_{L^\infty(\mathbb{R_{+}}\times\mathbb{R})}+1\big)^2}},
		\end{split}
	\end{equation}
	for all $t\in(0,T]$ and $x\in\mathbb{R}$, provided that  provided that $N\geq \max(c,c_T)$ and $k\geq c$, where the constants $c$ and $c_T$ are defined in \eqref{c} and \eqref{cT}, respectively.

	Next, noting that \eqref{sig-lower} remains valid, applying the second part of Proposition \ref{Prob6}, we get
	
	\begin{align*}
		\left\Vert \big[\sigma_{N+1}\big(s,u_{N+1}(s,y)\big)-\sigma_{N}\big(s,u_{N}(s,y)\big)\big]\mathbf{1}_{\Omega\setminus\mathcal{G}_{N+1}}(s,y)\right\Vert_k 
		 \leq 2K_0\exp{-\frac{e^{2N-4L_bs}}{(2k)eC_{\star}^2\big({\|u_0\|}_{L^\infty(\mathbb{R})}+{\|\sigma\|}_{L^\infty(\mathbb{R_{+}}\times\mathbb{R})}+1\big)^2}},
	\end{align*}
	uniformly for all $s\in(0,T]$, $y\in\mathbb{R}$, $N\geq 4\log C_0\vee C_{\#}^{\frac{2}{1-\beta/\alpha}}t\max\Big(2^{\frac{1}{1-\beta/\alpha}}L_\sigma^{\frac{2}{1-\beta/\alpha}}, L_b\Big)$, $k\geq\max\big(1,\frac{1}{2}L_b^{1-\beta/\alpha}L_{\sigma}^{-2}\big)$ and a positive constant $K_0$. This implies that
	
	\begin{equation}\label{I2-sgbDd}
		\begin{split}
			I_2^2
			\leq& \frac{8kC^{\star} \text{L}_{N,\sigma}^2\Gamma(1-\beta/\alpha)}{(2\gamma)^{1-\beta/\alpha}}e^{2\gamma t}\big[\mathcal{N}_{k, T,\gamma}(u_{N+1}-u_N)\big]^2\\
			&\qquad\qquad +\frac{16kK_0^2C^{\star}}{1-\beta/\alpha}\exp{-\frac{e^{2N-4L_bt}}{keC_{\star}^2\big({\|u_0\|}_{L^\infty(\mathbb{R})}+{\|\sigma\|}_{L^\infty(\mathbb{R_{+}}\times\mathbb{R})}+1\big)^2}}t^{1-\beta/\alpha},  
		\end{split} 
	\end{equation}
		for all $t\in(0,T]$ and $x\in\mathbb{R}$, provided that  provided that $N\geq \max(c,c_T)$ and $k\geq c$.
	
	Then, combining the bounds \eqref{I1-SgbDd} and \eqref{I2-sgbDd} leads to
	\begin{equation}\label{norm-sigCauchy}
		\begin{split}
			{\|u_{N+1}(t,x)-u_N(t,x)\|}_k&\leq  e^{\gamma t}\Bigg[\frac{\text{L}_{N,b}}{\gamma}+\frac{\sqrt{2^{2+\beta/\alpha}\Gamma(1-\beta/\alpha)k}\text{L}_{N,\sigma}}{\gamma^{\frac{1-\beta/\alpha}{2}}}\Bigg]\mathcal{N}_{k,\gamma,T}(u_{N+1}-u_N)\\
			&+ 2L_be^{2L_bt}C_{\star}\sqrt{k}t\Big({\|u_0\|}_{L^\infty(\mathbb{R})}+{\|\sigma\|}_{L^\infty(\mathbb{R_{+}}\times\mathbb{R})}+1\Big)\\
			&\qquad\qquad\qquad\times \exp{-\frac{e^{2N-4L_bt}}{2keC_{\star}^2\big({\|u_0\|}_{L^\infty(\mathbb{R})}+{\|\sigma\|}_{L^\infty(\mathbb{R_{+}}\times\mathbb{R})}+1\big)^2}}\\
			& +4K_0\sqrt{\frac{kC^{\star}}{1-\beta/\alpha}}\exp{-\frac{e^{2N-4L_bt}}{2keC_{\star}^2\big({\|u_0\|}_{L^\infty(\mathbb{R})}+{\|\sigma\|}_{L^\infty(\mathbb{R_{+}}\times\mathbb{R})}+1\big)^2}}t^{\frac{1-\beta/\alpha}{2}},  
		\end{split}
	\end{equation}
	provided that  $N\geq \max(c,c_T)$ and $k\geq c$.

	Next, making the same specific choices of $k$, $\gamma$ and $A_0$ as in \eqref{k-and-g} and \eqref{A0}, we get
	
	\begin{align*}
		\sup\limits_{x\in\mathbb{R}}{\|u_{N+1}(t,x)-u_N(t,x)\|}_c&\leq e^{\gamma t}\Bigg[\frac{\text{L}_{N,b}}{16\big[C^{\star}\Gamma(1-\beta/\alpha)\big]^{\frac{1}{1-\beta/\alpha}}A_0^{\frac{2}{1-\beta/\alpha}}c^{\frac{1}{1-\beta/\alpha}} \text{L}_{N,\sigma}^{\frac{2}{1-\beta/\alpha}}}+\frac{1}{A_0}\Bigg]\mathcal{N}_{c,\gamma,T}(u_{N+1}-u_N)\\\\
		& +\Bigg[2\sqrt{c}C_{\star}L_be^{2L_bt}t\big({\|u_0\|}_{L^\infty(\mathbb{R})}+{\|\sigma\|}_{L^\infty(\mathbb{R_{+}}\times\mathbb{R})}+1\big)
		+4K_0\sqrt{\frac{kC^{\star}}{1-\beta/\alpha}}t^{\frac{1-\beta/\alpha}{2}}
		\Bigg]\\
		&\qquad\qquad\times\exp{-\frac{e^{2N-4L_bt}}{2ceC_{\star}^2\big({\|u_0\|}_{L^\infty(\mathbb{R})}+{\|\sigma\|}_{L^\infty(\mathbb{R_{+}}\times\mathbb{R})}+1\big)^2}}\\
		&\leq e^{\gamma t}\Bigg[\frac{1}{16}+\frac{1}{4}\Bigg]\mathcal{N}_{c,\gamma,T}(u_{N+1}-u_N)\\
		& +\Bigg[2\sqrt{c}C_{\star}L_be^{2L_bt}t\big({\|u_0\|}_{L^\infty(\mathbb{R})}+{\|\sigma\|}_{L^\infty(\mathbb{R_{+}}\times\mathbb{R})}+1\big)
		+4K_0\sqrt{\frac{kC^{\star}}{1-\beta/\alpha}}t^{\frac{1-\beta/\alpha}{2}}
		\Bigg]\\
		&\qquad\qquad\times\exp{-\frac{e^{2N-4L_bt}}{2ceC_{\star}^2\big({\|u_0\|}_{L^\infty(\mathbb{R})}+{\|\sigma\|}_{L^\infty(\mathbb{R_{+}}\times\mathbb{R})}+1\big)^2}}\\
	\end{align*}
	uniformly for all $T>0, t\in(0,T],$ and $N\geq \max(N_0,c,c_T);$ Again, since $1/16+1/4<1/2,$ dividing both sides of the inequality above by $e^{\gamma t}$ and optimizing over $t\in(0,T]$, we get
	\begin{equation}\label{N-bounded-sg}
		\begin{split}
			\mathcal{N}_{c,\gamma,T}(u_{N+1}-u_N)\leq
			&2\Bigg[\sqrt{c}C_{\star}L_be^{2L_bT}T\big({\|u_0\|}_{L^\infty(\mathbb{R})}+{\|\sigma\|}_{L^\infty(\mathbb{R_{+}}\times\mathbb{R})}+1\big)
			+2K_0\sqrt{\frac{kC^{\star}}{1-\beta/\alpha}}T^{\frac{1-\beta/\alpha}{2}}
			\Bigg]\\
			&\qquad\qquad\times\exp{-\frac{e^{2N-4L_bT}}{2ceC_{\star}^2\big({\|u_0\|}_{L^\infty(\mathbb{R})}+{\|\sigma\|}_{L^\infty(\mathbb{R_{+}}\times\mathbb{R})}+1\big)^2}},\\
		\end{split}
	\end{equation}
	uniformly for all $T\in(0,T_0)$ and $N\geq \max(N_0,c,c_T)$. Note that we have used in \eqref{N-bounded-sg} the fact that $$t\mapsto\frac{e^{2N-4L_bt}}{2ceC_{\star}^2\big({\|u_0\|}_{L^\infty(\mathbb{R})}+{\|\sigma\|}_{L^\infty(\mathbb{R_{+}}\times\mathbb{R})}+1\big)^2}$$ is non-increasing.

 We now apply Proposition \ref{lm:2.5} in order to deduce from \eqref{N-bounded-sg} and \eqref{NT-norm} that, for every $T>0$ fixed,
	\begin{equation}\label{limSup-bd}
		\limsup\limits_{N\rightarrow\infty}e^{-2N}\log\sup\limits_{t\in(0,T)}\sup\limits_{x\in\mathbb{R}}{\|u_{N+1}(t,x)-u_N(t,x)\|}_{k}<0 \qquad\forall k\geq 1.
	\end{equation}
	And from there,  this proves that 
	\begin{equation}\label{sum-Lk-b}
		\sum\limits_{N=1}^\infty\sup\limits_{t\in(0,T]}   \sup\limits_{x\in\mathbb{R}}{\|u_{N+1}(t,x)-u_N(t,x)\|}_{k} <\infty \qquad \forall k\geq 1.
	\end{equation}
	which in turn proves that
	\begin{equation}\label{lim-uN-b}
		u(t,x)=\lim\limits_{N\rightarrow\infty}u_N(t,x) \ \text{exists in } \ L^{k}(\Omega),
	\end{equation}
	and the rate of convergence does not depend on $t\in(0,T]$ nor $x\in\mathbb{R}.$ As a direct consequence, $u$ is $L^k-$continuous.

	For Step 2,  we use the second parts of Propositions \ref{Prob4} and \ref{lm:2.5} to ensure that \eqref{lim-Tail} continues to hold. Therefore, \eqref{lim-difIbN}, \eqref{lim-Ib}, \eqref{I-sg-N}, \eqref{Int-eq-I-sg} all remain true. Thus, \eqref{last-Int} also holds for all $t\in(0,T]$  and $x\in\mathbb{R}.$ Finally, the triangle inequality yields \eqref{lim-Ib} which proves the second assertion of \eqref{lim-IN} and completes this proof.
	\subsection{Proof of uniqueness}
	Let $u,v$ be two solutions of \eqref{MainEq} with the same initial condition $u_0$ satisfying Assumption \ref{As1}. Recall that $u$ and $v$ satisfy the mild solution formulation \eqref{MildSol} with the functions $b$ and $\sigma$ satisfying Assumption \ref{As2}.  Next, using the Burkh\"older-Davis-Gundy inequality, we find that
	\begin{equation}\label{U-V}
		{\|u(t,x)-v(t,x)\|}_k\leq I_1+I_2+I_3+I_4 \ \text{for all} \ t>0, x\in\mathbb{R} \ \text{and} \ k\geq 1,
	\end{equation}
	where
	\begin{equation}\label{I1}
		I_1=\int\limits_0^t\int\limits_{\mathbb{R}}G_{t-s}(y-x)\left \Vert \Big[b\big(s,u(s,y)\big)-b\big(s,v(s,y)\big)\Big]\mathbf{1}_{A_N(s,y)}\right\Vert_k ds\ dy, 
	\end{equation}
	\begin{equation}\label{I2}
		I_2^2=8k\int\limits_0^t\int\limits_{\mathbb{R}}\big[G_{t-s}(y-x)\big]^2\left\Vert\Big[\sigma\big(s,u(s,y)\big)-\sigma\big(s,v(s,y)\big)\Big]\mathbf{1}_{A_N(s,y)}\big]   \right\Vert_k^2ds\ dy,
	\end{equation}
	\begin{equation}\label{I3}
		I_3=\int\limits_0^t\int\limits_{\mathbb{R}}G_{t-s}(y-x)\left \Vert \Big[b\big(s,u(s,y)\big)-b\big(s,v(s,y)\big)\Big]\mathbf{1}_{\Omega\setminus A_N(s,y)}\right\Vert_k ds\ dy, 
	\end{equation}
	and 
	\begin{equation}\label{I4}
		I_4^2=8k\int\limits_0^t\int\limits_{\mathbb{R}}\big[G_{t-s}(y-x)\big]^2\left\Vert\Big[\sigma\big(s,u(s,y)\big)-\sigma\big(s,v(s,y)\big)\Big]\mathbf{1}_{\Omega\setminus A_N(s,y)}\big]   \right\Vert_k^2ds\ dy,
	\end{equation}
	and for $N,s>0$ and $y\in\mathbb{R},$
	$$
	A_N(s,y)=\{\omega\in\Omega: |u(s,y)|(\omega)\leq e^N, |v(s,y)|(\omega)\leq e^N\}.
	$$
	
Note that the moment bounds obtained in Proposition \ref{Prob4} only use the linear growth constants $L_\sigma$ and $L_b$. Therefore, they continue to hold when $u_N$ is replaced by $u$ or $v.$ Moreover, the tail estimates obtained in Proposition \ref{Prob6} also hold when $u_{N+1}$ is replaced by $u$ or $v$.
Therefore, we can proceed as in the proof of existence but appeal to the local Lipschitz conditions on $b$ and $\sigma$, i.e we recall \eqref{LN} and write 
\begin{align*}
	\left \Vert \Big[b\big(s,u(s,y)\big)-b\big(s,v(s,y)\big)\Big]\mathbf{1}_{A_N(s,y)}\right\Vert_k\leq \text{L}_{N,b} \left \Vert u(s,y)-v(s,y)\right\Vert_k
\end{align*}
and 
\begin{align*}
	\left \Vert \Big[\sigma\big(s,u(s,y)\big)-\sigma\big(s,v(s,y)\big)\Big]\mathbf{1}_{A_N(s,y)}\right\Vert_k\leq \text{L}_{N,\sigma} \left \Vert u(s,y)-v(s,y)\right\Vert_k.
\end{align*}
	\subsubsection{Proof of uniqueness when $L_\sigma>0$}

	Since $\mathbb{P}\big(\Omega\setminus A_N(s,y)\big)\leq \mathbb{P}\big(|u(t,x)|\geq e^N\big)+\mathbb{P}\big(|v(t,x)|\geq e^N\big)$, similar computations to those in the proof of existence lead to, for all $T>0$,
	\begin{align*}
			&{\|u(t,x)-v(t,x)\|}_k\leq e^{\gamma t}\Big[\frac{\text{L}_{N,b}}{\gamma}+\frac{2^{1+\beta/(2\alpha)}\sqrt{kC^{\star}\Gamma(1-\beta/\alpha)}\text{L}_{N,\sigma}}{\gamma^{\frac{1-\beta/\alpha}{2}}}\Big]\mathcal{N}_{k,\gamma,T}(u_{N+1}-u_N)\\
			& \qquad+\frac{2C_0L_b}{4(C_{\#}L_\sigma)^{\frac{2}{1-\beta/\alpha}}k^{\frac{1}{1-\beta/\alpha}}}\exp{4(C_{\#}L_\sigma)^{\frac{2}{1-\beta/\alpha}}k^{\frac{1}{1-\beta/\alpha}} t-\frac{N^{2-\beta/\alpha}}{2k(C_{\#}L_\sigma)^2 (8t)^{1-\beta/\alpha}}}\\
			&\qquad+ \frac{4\sqrt{kC^{\star}}C_0L_\sigma}{\sqrt{1-\beta/\alpha}}\exp{4(C_{\#}L_\sigma)^{\frac{2}{1-\beta/\alpha}}k^{\frac{1}{1-\beta/\alpha}} t}\exp{-\frac{N^{2-\beta/\alpha}}{2k(C_{\#}L_\sigma)^2 (8t)^{1-\beta/\alpha}}}t^{\frac{1-\beta/\alpha}{2}}
	\end{align*}
	uniformly for all $\gamma>0, t\in(0,T], x\in\mathbb{R,}$ and  $N\geq\max(c,c_T)$ and $k\geq c$, where $C_0$ and $C_{\#}$ were defined in Proposition \ref{Prob4}, $c$ and $c_T$ were defined in \eqref{cT} and \eqref{c}, respectively. Then choosing the parameters $k$ and $\gamma$ as in \eqref{k-and-g} and proceeding exactly as in the proof of existence, we find that 
	\begin{equation*}
		\sum\limits_{N=1}^\infty\sup\limits_{t\in(0,T]}   \sup\limits_{x\in\mathbb{R}}{\|u(t,x)-v(t,x)\|}_{2} <\infty \qquad \forall T> 0.
	\end{equation*}
	But $\sup\limits_{t\in(0,T]} \sup\limits_{x\in\mathbb{R}}{\|u(t,x)-v(t,x)\|}_{2}$ does not depend on $N$, therefore it must be the case that $$\sup\limits_{t\in(0,T]}   \sup\limits_{x\in\mathbb{R}}{\|u(t,x)-v(t,x)\|}_{2}=0$$ and this concludes the proof.

\subsubsection{Proof of uniqueness when $\sigma\in L^\infty(\mathbb{R}_+\times\mathbb{R})$}
	The proof in this case is also very similar to the case $L_\sigma>0$ with minor modifications. Similar computations to those in the proof of existence lead to 
	
	\begin{align*}
		{\|u(t,x)-v(t,x)\|}_k&\leq  e^{\gamma t}\Bigg[\frac{\text{L}_{N,b}}{\gamma}+\frac{\sqrt{2^{2+\beta/\alpha}\Gamma(1-\beta/\alpha)k}\text{L}_{N,\sigma}}{\gamma^{\frac{1-\beta/\alpha}{2}}}\Bigg]\mathcal{N}_{k,\gamma,T}(u_{N+1}-u_N)\\
		&+ 2L_be^{2L_bt}C_{\star}\sqrt{k}t\Big({\|u_0\|}_{L^\infty(\mathbb{R})}+{\|\sigma\|}_{L^\infty(\mathbb{R_{+}}\times\mathbb{R})}+1\Big)\\
		&\qquad\qquad\qquad\times \exp{-\frac{e^{2N-4L_bt}}{2keC_{\star}^2\big({\|u_0\|}_{L^\infty(\mathbb{R})}+{\|\sigma\|}_{L^\infty(\mathbb{R_{+}}\times\mathbb{R})}+1\big)^2}}\\
		& +4K_0\sqrt{\frac{kC^{\star}}{1-\beta/\alpha}}\exp{-\frac{e^{2N-4L_bt}}{2keC_{\star}^2\big({\|u_0\|}_{L^\infty(\mathbb{R})}+{\|\sigma\|}_{L^\infty(\mathbb{R_{+}}\times\mathbb{R})}+1\big)^2}}t^{\frac{1-\beta/\alpha}{2}},  
	\end{align*}
	
which is valid	uniformly for all $\gamma>0, t\in(0,T], x\in\mathbb{R,}$ and  $N\geq\max(c,c_T)$ and $k\geq c$, where $C_{\star}$ was defined in Proposition \ref{Prob4}, $c$ and $c_T$ were defined in \eqref{cT} and \eqref{c}, respectively.
Again,  choosing the parameters $k$ and $\gamma$ as in \eqref{k-and-g} and proceeding exactly as in the proof of existence, we find that 
	\begin{equation*}
		\sum\limits_{N=1}^\infty\sup\limits_{t\in(0,T]}   \sup\limits_{x\in\mathbb{R}}{\|u(t,x)-v(t,x)\|}_{2} <\infty \qquad \forall T> 0.
	\end{equation*}
	Again, because $\sup\limits_{t\in(0,T]}   \sup\limits_{x\in\mathbb{R}}{\|u(t,x)-v(t,x)\|}_{2}$ does not depend on $N$, it must be the case that $$\sup\limits_{t\in(0,T]} \sup\limits_{x\in\mathbb{R}}{\|u(t,x)-v(t,x)\|}_{2}=0.$$ This finally completes the proof of uniqueness. 


	\Addresses
	
\end{document}